\documentclass[12pt]{article}
\usepackage{amsmath,amssymb,amsthm}
\usepackage[margin=1.25in]{geometry}
\usepackage[bookmarksnumbered]{hyperref}
\usepackage{enumitem,theoremref,verbatim,cancel, breqn, fancyref}

\title{The weak Bernoulli property for matrix Gibbs states}
\author{Mark Piraino \thanks{Department of Mathematics and Statistics, University of Victoria.}}
\date{\today}

\theoremstyle{definition}
\newtheorem{theorem}{Theorem}[section]
\newtheorem{lemma}[theorem]{Lemma}
\newtheorem{example}[theorem]{Example}
\newtheorem{corollary}[theorem]{Corollary}
\newtheorem{proposition}[theorem]{Proposition}
\newtheorem{definition}[theorem]{Definition}

\newtheorem*{acknowledgements}{Acknowledgements}

\newcommand{\inn}[1]{\left\langle #1 \right\rangle}
\newcommand{\set}[1]{\left\{ #1 \right\}}
\newcommand{\abs}[1]{\left| #1 \right|}

\newcommand{\norm}[1]{\left \| #1 \right \|}

\newcommand{\wh}[1]{\widehat{ #1 }}
\newcommand{\ol}[1]{\overline{#1}}

\def\[#1\]{\begin{align*}#1\end{align*}}

\DeclareMathOperator{\spn}{span}

\DeclareMathOperator{\var}{var}
\DeclareMathOperator{\tr}{tr}

\DeclareMathOperator{\interior}{int}

\newcommand{\Q}{\mathcal{Q}}
\newcommand{\R}{\mathbb{R}}

\newcommand{\Z}{\mathbb{Z}}
\newcommand{\RP}{\mathbb{R}\mathbb{P}}
\def\M{\mathcal{M}}     
\def\A{\mathcal{A}}    
\def\P{\mathcal{P}}
\def\H{\mathcal{H}}

\def\S{\Sigma}  
  
\def\r{\mathcal{R}}

\newcommand{\e}{\varepsilon}

\begin{document}

\maketitle

\begin{abstract}
	We study the ergodic properties of a class of measures on $\Sigma^{\mathbb{Z}}$ for which $\mu_{\mathcal{A},t}[x_{0}\cdots x_{n-1}]\approx e^{-nP}\left \|A_{x_{0}}\cdots A_{x_{n-1}}\right \| ^{t}$, where $\mathcal{A}=(A_{0}, \ldots , A_{M-1})$ is a collection of matrices. The measure $\mu_{\mathcal{A},t}$ is called a matrix Gibbs state. In particular we give a sufficient condition for a matrix Gibbs state to have the weak Bernoulli property. We employ a number of techniques to understand these measures including a novel approach based on Perron-Frobenius theory. We find that when $t$ is an even integer the ergodic properties of $\mu_{\mathcal{A} ,t}$ are readily deduced from finite dimensional Perron-Frobenius theory. We then consider an extension of this method to $t>0$ using operators on an infinite dimensional space. Finally we use a general result of Bradley to prove the main theorem.
\end{abstract}

\section{Introduction}

We recall the definition of a scalar  Gibbs state. Let $\Sigma_{A}$ be a shift of finite type and $\varphi:\Sigma_{A}\to \R$. We say that a shift invariant measure, $\mu_{\varphi}$, is a scalar Gibbs state for $\varphi$ provided there exists $C>0$ and $P$ such that
\[C^{-1}\leq \frac{\mu_{\varphi}([x_{0}\cdots x_{n-1}])}{e^{-nP+S_{n}\varphi}}\leq C \]
for all $x \in \Sigma_{A}$ and $n>0$ (where $S_{n}\varphi=\sum_{k=0}^{n-1}\varphi(\sigma^{k}x)$). By analogy if $\A=(A_{0},\ldots, A_{M-1})\in M_{d}(\R)^{M}$ and $t>0$ we say that a shift invariant measure $\mu_{\A , t}$ is a matrix Gibbs state for $(\A , t)$ provided there exists a constant $C>0$ and $P$ such that
\begin{equation}\label{eq:GibbsInequalityMatrices}
C^{-1}\mu_{\A , t}([x_{0}\cdots x_{n-1}])\leq e^{-nP}\norm{A_{x_{0}}\cdots A_{x_{n-1}}}^{t}\leq C \mu_{\A, t}([x_{0}\cdots x_{n-1}])
\end{equation}
for all $x \in \Sigma^{\Z}$ ($\Sigma=\set{0, \ldots , M-1}$) and $n>0$. Notice we are working with the two-sided shift and not as has been done in previous literature the one-sided shift. Thus in a strict sense one may consider that we are working with the invertible extension of matrix Gibbs states, this is important when working on the isomorphism problem and it is also necessary so that we can apply the results in \cite{MR684493}. When $t=1$ we refer to the measure simply as the Gibbs state for $\A$. $P$ is uniquely determined by \eqref{eq:GibbsInequalityMatrices} and is called the pressure denoted $P(\A , t)$. A computation shows that 
\[ P(\A,t)=\lim_{n \to \infty}\frac{1}{n}\log \left(\sum_{x_{0}\cdots x_{n-1}}\norm{A_{x_{0}}\cdots A_{x_{n-1}}}^{t}\right). \]
For the remainder of the article a Gibbs state will always refer to a matrix Gibbs state. Matrix Gibbs states are also equilibrium states for a sub-additive variational principle \cite{MR2373208}
\begin{equation}\label{eq:variationalprinciple}
P(\A ,t)=\sup_{\mu\in \M(\sigma)}\left[h(\mu)+t\Lambda(\A,\mu)\right].
\end{equation}
where $\Lambda(\A ,\mu)$ is the maximal Lyapunov exponent 
\[ \Lambda(\A , \mu)=\lim_{n \to \infty}\frac{1}{n}\int \log \norm{A_{x_{0}}\cdots A_{x_{n-1}}}d \mu(x). \]

Measures which achieve the supremum are called matrix equilibrium states. Such measures always exist by weak$^{\ast}$ compactness and upper semi-continuity of $h(\mu)+t \Lambda(\A,\mu)$. The connection between Gibbs states and equilibrium states for the variation principle \eqref{eq:variationalprinciple} was studied in \cite{MR2784616}. The study of these measures was originally motivated by their applications to dimension theory \cite{MR1909650}. However recently interest has been shown in determining their ergodic properties \cite{morris_2017} \cite{morris_2017_2}. In the classical case for H\"older continuous functions, scalar Gibbs states are well known to have many nice statistical properties. It is natural to ask to what extent matrix Gibbs states share these properties. 

One of the strongest of these properties is that the dynamical system defined by the shift map and a scalar Gibbs state for a H\"older potential is isomorphic to a Bernoulli shift and this is the problem we will focus on this article. This is a particularly appealing property because Bernoulli shifts are classified up to isomorphism by their entropy \cite{MR0257322}. In general it is very difficult to explicitly construct isomorphisms between measure preserving systems. One of the most common methods for demonstrating a measure preserving system is isomorphic to a Bernoulli shift is to show that it is weak Bernoulli and appeal to \cite{MR0274718}. This is the strategy we will take in this paper. The same method has been used by Bowen \cite{MR0387539} for scalar Gibbs states. Recall what it means for a dynamical system to be weak Bernoulli.

\begin{definition}
	We say that partitions $\Q$ and $\r$ are $\e$-independent (written $\Q\perp^{\e}\r$) if 
	\[ \sum_{q\in \Q, r\in \r}\abs{\mu(q\cap r)-\mu(q)\mu(r)}<\e. \]
	We say that a partition $\P$ is weak Bernoulli if for every $\e>0$ there exists $N$ such that $\bigvee_{i=0}^{s-1}\sigma^{-i}\P\perp^{\e}\bigvee_{i=t}^{t+r-1}\sigma^{-i}\P$ for all $r,s\geq 0$ and $t \geq s+N$. We say that $\mu_{\A,t}$ is weak Bernoulli if the standard partition $\P=\set{[i]:0 \leq i \leq M-1}$ is weak Bernoulli.
\end{definition}
For a word $I=i_{0}i_{1} \cdots i_{n-1}$ we write
\[ A_{I}:=A_{i_{0}}A_{i_{1}}\cdots A_{i_{n-1}} \]
and we denote the length of the word $I$ by $\abs{I}$. We say that $\A=(A_{0},  \ldots, A_{M-1})\in M_{d}(\R)^{M}$ is \emph{irreducible} if the matrices have no common proper and non-trivial invariant subspace. This implies that there exists a constant $\delta>0$ such that
\begin{equation}\label{eq:irrinequality}
 \sum_{\abs{K}\leq d}\norm{A_{I}A_{K}A_{J}}\geq \delta \norm{A_{I}} \norm{A_{J}} 
\end{equation}
for all $I,J$. With this in mind we make the following definition

\begin{definition}
	We say that $\A=(A_{0},\ldots, A_{M-1})$ is \emph{primitive} if there exists an $N$ and a $\delta>0$ such that
	\begin{equation}\label{eq:priminequality}
	 \sum_{\abs{K}=N}\norm{A_{I}A_{K}A_{J}}\geq \delta \norm{A_{I}}\norm{A_{J}} 
	 \end{equation}
	for all $I,J$.
\end{definition}

For both irreducible and primitive collections of matrices, matrix Gibbs states are known to exist and be unique \cite[theorem 5.5]{MR2739782} for all $t>0$. The terms irreducible and primitive are familiar from Perron-Frobenius theory and indeed the notions are connected. Let $L_{\A}:M_{d}(\R)\to M_{d}(\R)$ be defined by $L_{\A}B=\sum_{i}A_{i}^{\ast}BA_{i}$, then $L_{\A}$ preserves the cone of positive semi-definite matrices. The operator $L_{\A}$ appears in connection with a class of measures related to fractal geometry called Kusuoka measures \cite{MR1025071} (see example \ref{ClassicalKusuokaMeasure}). One can check that if $L_{\A}$ is irreducible (respectively primitive) in the sense of Perron-Frobenius theory then $\A$ satisfies equation \eqref{eq:irrinequality} (respectively equation \eqref{eq:priminequality}). For the details see proposition \ref{IrrPrimDefProp}. Our main theorem is the following.

\begin{theorem}\label{MainThm}
	Suppose that $\A=(A_{0}, \ldots , A_{M-1})$ is primitive. Then for any $t>0$ the unique $t$-Gibbs state for $\A$ is weak Bernoulli.
\end{theorem}

The proof of theorem \ref{MainThm} can be found in section \ref{sec:GeneralMatrices}. The proof relies on a general result of Bradley \cite{MR684493}, which is somewhat opaque. With this in mind we also present a method for understanding matrix Gibbs states through transfer operators which is interesting in its own right. Understanding the ergodic/statistical properties of Gibbs states in sub-additive thermodynamic formalism has long been a challenge, with most results being achieved using fairly ad-hoc methods. This is in contrast to the case for scalar Gibbs states which has a well developed methodology for deducing ergodic/statistical properties relying on the transfer operator. In this article we adapt the classical doctrine of transfer operators for scalar Gibbs states to matrix Gibbs states.

In section \ref{sec:MatricesCommonCone} we show that in the case when $t$ is an even integer the ergodic properties of $\mu_{A ,t}$ can be readily understood by studying the convergence properties of a matrix. As a consequence we can obtain an exponential mixing result which includes an explicit rate determined by the spectral gap of a finite dimensional matrix. This naturally leads to the problem of generalizing this approach to $t>0$. In section \ref{sec:fastmixing} we generalize section \ref{sec:MatricesCommonCone} using operators on a suitable infinite dimensional vector space. A major advantage of the approach in sections \ref{sec:MatricesCommonCone} and \ref{sec:fastmixing} is that we can give an explicit construction of certain Gibbs states, including a formula for the measure of a cylinder set. Previous methods have relied on abstract compactness arguments, realizing the Gibbs state as a weak$^{\ast}$ limit point of a sequence of measures. As many properties are not preserved under weak$^{\ast}$ limits this makes an analysis of the Gibbs state difficult. Our transfer operator approach allows us to give direct proofs of ergodic properties. It also provides a strong intuition for understanding how properties of the collection $\A$ are reflected in the ergodic properties of $\mu_{\A,t}$.

\section{Matrices which preserve a common cone}\label{sec:MatricesCommonCone}

One particular class of matrix Gibbs states has appeared extensively in applications. Consider the following examples.

\begin{example}
	Bernoulli measures, take $d=1$.
\end{example}

\begin{example}
	Factors of Markov measures. The $1$-Gibbs states for collections of non-negative matrices are precisely factors of Markov measures, for details see \cite{MR2866664} or \cite{chazottes2003projection}, \cite{Yoo2010}. In fact, allowing the operators in $\A$ to act on an infinite dimensional space, factors of Gibbs states for H\"older potentials can be viewed as Gibbs states for a suitable collection of operators, see \cite{HolderGibbsStates}.
\end{example}

\begin{example}\label{ClassicalKusuokaMeasure}
	The Kusuoka measure \cite{MR1025071} was originally studied because of its connections to fractal geometry. We briefly recall the construction. Let $L_{i}B= A_{i}^{\ast}BA_{i}$ and $L_{\A}=\sum_{i}L_{i}$. When $\A$ is irreducible there exist $U,V$ positive definite matrices such that $L_{\A}U=\rho(L_{\A})U$, $L_{\A}^{\ast}V=\rho(L_{\A})V$ (notice that $L_{\A}^{\ast}B= \sum_{i}A_{i}BA_{i}^{\ast}$) and $\inn{U,V}_{\text{HS}}=1$ (where $\inn{A,B}_{\text{HS}}=\tr(A^{\ast}B)$). The Kusuoka measure is then obtained by extending
	\[ \mu[x_{0}\cdots x_{n-1}]=\rho(L_{\A})^{-n}\inn{L_{x_{0}}L_{x_{1}}\cdots L_{x_{n-1}}U,V}_{\text{HS}} \] 
	to a measure using Carath\'eodory's extension theorem. It was shown in \cite{morris_2017} that the Kusuoka measure is a $2$-Gibbs state. We will generalize this result to $k$-Gibbs states for $k$ even in example \ref{GeneralizedKusoukaMeasure}. Observe that thinking of the linear maps $L_{i}$ as matrices we have that the Kusuoka measure is the $1$-Gibbs state for the collection $\wh{\A}=(L_{0}, \ldots , L_{M-1})$ each of which preserves the cone of positive semi-definite matrices.
\end{example}

The property shared by all of these matrix equilibrium states is that all of the matrices preserve a common cone. Our goal for this section is then to treat these measures in an abstract manner. As one of the applications of this section is the Kusuoka measure, we work with matrices preserving an abstract cone $K$. For the most part, the reader will lose no intuition by simply thinking of $K$ as being the positive quadrant of $\R^{d}$. For the reader's convenience we have collected some definitions and facts about abstract cones in finite dimensional vector spaces in the appendix. Recall that 
\[ \var_{n}f=\sup\set{\abs{f(x)-f(y)}: x_{i}=y_{i} \text{ for all }\abs{i} \leq n-1} \]
and for $\theta \in (0,1)$ define
\[ \H_{\theta}=\set{f \in C(\Sigma^{\Z}):\text{There exists a constant }K>0\text{ for which }\var_{n}f\leq K \theta^{n}}. \]
We denote the least such constant by $\abs{f}_{\theta}$ and $\H_{\theta}$ becomes a Banach space with norm $\norm{f}_{\theta}=\norm{f}_{\infty}+\abs{f}_{\theta}$. The goal of this section is to prove the following theorem.

\begin{theorem}\label{Section2thm}
	Let $\A=(A_{0},\ldots, A_{M-1})\in M_{d}(\R)^{M}$. Suppose that each $A_{i}$ is non-negative with respect to a cone $K$ and $A:=\sum_{i}A_{i}$ is such that $\sum_{k=0}^{d-1}A^{k}$ maps $K\setminus \set{0}$ into the interior of $K$ (that is $A$ is $K$-irreducible). Then there exists a $1$-Gibbs state for $\A$ denoted $\mu_{\A}$ moreover
	\begin{enumerate}
		\item 
		$\mu_{\A}$ is ergodic and thus unique, and $P(\A,1)=\log \rho(A)$.
		
		\item
		If there exists an $N$ such that $A^{N}$ maps $K\setminus \set{0}$ into the interior of $K$ (that is $A$ is $K$-primitive) then 
		\begin{enumerate}
			\item 
			$\mu_{\A}$ is weak Bernoulli.
			
			\item
			$\mu_{\A }$ has exponential decay of correlations for H\"older continuous functions. That is for a fixed $\theta \in (0,1)$ there are constants $D$ and $\gamma \in (0,1)$ such that
			\[ \abs{\int f\cdot g \circ \sigma^{n}d\mu_{\A} - \int f d\mu_{\A} \int gd\mu_{\A}}\leq D\norm{f}_{\theta}\norm{g}_{\theta}\gamma^{n} \]
			for all $f,g\in \H_{\theta}$, $n \geq 0$. In addition, the rate $\gamma$ is determined by $\theta$ and the eigenvalues of $A$.
		\end{enumerate}
	\end{enumerate}  
\end{theorem}

For the Kusuoka measure, part 2(b) is known \cite{MR3667706}, however our proof is fundamentally different and significantly more elementary. In particular the method in \cite{MR3667706} uses the $g$-function for the Kusuoka measure and transfer operator techniques. This is technically challenging largely due to the fact that the $g$-function can fail to be continuous.

We can explicitly construct the measure $\mu_{\A}$. As $A$ is irreducible we may take $u,v$ to be right and left eigenvectors respectively corresponding to the spectral radius $\rho(A)$ with $\inn{u,v}=1$. On cylinder sets we define
\begin{equation}\label{eq:DefMeasure}
\mu_{\A}[x_{0}x_{1}\cdots x_{n-1}]=\rho(A)^{-n}\inn{A_{x_{0}}A_{x_{1}}\cdots A_{x_{n-1}} u ,v}. 
\end{equation}
Using the fact that $u,v$ are eigenvectors for $A$ it is readily checked that
\[ \sum_{i}\mu_{\A}[ix_{0}\cdots x_{n-1}]=\mu_{\A}[x_{0}\cdots x_{n-1}]=\sum_{i}\mu_{\A}[x_{0}\cdots x_{n-1}i]. \]
As cylinder sets form a semi-algebra Carath\'eodory's extension theorem implies that this extends to a shift invariant measure on $\S^{\Z}$. Next our goal is to show that this is a 1-Gibbs state for $\A$ and that it is unique. To do so, we prove the following proposition.

\begin{proposition}\label{BasicPropertiesMesaureFullShift}
	Suppose that $\A=(A_{0},\ldots, A_{M-1})\in M_{d}(\R)^{M}$ is such that each $A_{i}$ is non-negative with respect to a cone $K$ and $A:=\sum_{i}A_{i}$ is $K$-irreducible. Then
	\begin{enumerate}
		\item 
		$\mu_{\A}$ is ergodic.
		
		\item
		$\mu_{\A}$ satisfies the Gibbs inequality \eqref{eq:GibbsInequalityMatrices} with $P=\log \rho(A)$.
	\end{enumerate}
	
	\begin{proof}
		\begin{enumerate}
			\item 
			Observe that
			\begin{equation}\label{eq:powersofsum}
			A^{n}=\left(\sum_{i}A_{i}\right)^{n}= \sum_{\abs{K}=n}A_{K}. 
			\end{equation}
			Let $I,J$ be words.
			\begin{align*}
			&\abs{\frac{1}{n}\sum_{k=1}^{n}\mu_{\A}([I] \cap \sigma^{-k}[J])-\mu_{\A}([I])\mu_{\A}([J])}\\
			&\leq\abs{\frac{1}{n}\sum_{k=1}^{\abs{I}}\mu_{\A}([I] \cap \sigma^{-k}[J])}\\
			&\phantom{=}\,+\rho(A)^{-\abs{I}-\abs{J}} \abs{\inn{A_{I}\left(\frac{1}{n}\sum_{k=\abs{I}+1}^{n}\rho(A)^{\abs{I}-k}A^{k-\abs{I}}\right)A_{J}u,v}-\inn{A_{I}u,v}\inn{A_{J}u,v}}\\
			&\xrightarrow{n\to \infty}0+\rho(A)^{-\abs{I}-\abs{J}}\abs{\inn{A_{I}\inn{A_{J}u,v}u,v}-\inn{A_{I}u,v}\inn{A_{J}u,v}}=0
			\end{align*}
			 by the Perron-Frobenius theorem \ref{PFThmCones} 2(b). As cylinder sets are a generating semi-algebra this implies that $\mu_{\A}$ is ergodic.
			
			\item
			From the Perron-Frobenius theorem we have that $u \in \interior(K)$, $v \in \interior(K^{\ast})$. Thus the Gibbs inequality follows directly from an application of lemma \ref{approxnormlemma}.
		\end{enumerate}
	\end{proof}
\end{proposition}

As ergodic measures are mutually singular this implies that $\mu_{\A}$ is the unique $1$-Gibbs state for $\A$. The proof of the previous lemma shows that mixing properties of $\mu_{\A}$ are related to the convergence of $A^{n}$. It is this fact that we will exploit to prove the remaining assertions in theorem \ref{Section2thm}.

\begin{proposition}\label{weakbernoulli}
	Suppose that $\A=(A_{0},\ldots, A_{M-1})\in M_{d}(\R)^{M}$ is such that each $A_{i}$ is non-negative with respect to a cone $K$ and $A:=\sum_{i}A_{i}$ if $A$ is $K$-primitive then the measure $\mu_{\A}$ is weak Bernoulli. 
	\begin{proof}
		Let $r,s\geq 1$, $t \geq s$ and take $[I] \in \bigvee_{i=0}^{s-1}\sigma^{-i}\P$ and $[_{t}J]\in \bigvee_{i=t}^{t+r-1}\sigma^{-i}\P$. Notice
		\begin{align*}
		&\abs{\mu_{\A}([I] \cap [_{t}J])-\mu_{\A}([I])\mu_{\A}([J])}\\
		&=\abs{\sum_{\abs{K}=t-s}\mu_{\A}([IKJ])-\mu_{\A}([I])\mu_{\A}([J])}\\
		&=\abs{\sum_{\abs{K}=t-s}\rho(A)^{-(s+r+(t-s))}\inn{A_{I}A_{K}A_{J}u,v}-\rho(A)^{-(s+r)}\inn{A_{I}u,v}\inn{A_{J}u,v}}\\
		&=\rho(A)^{-(s+r)}\abs{\inn{A_{I}\left(\rho(A)^{-(t-s)}\sum_{\abs{K}=t-s}A_{K}\right)A_{J}u,v}-\inn{A_{I}u,v}\inn{A_{J}u,v}} 
		\end{align*}
		Notice that 
		\[ \rho(A)^{-(t-s)}\sum_{\abs{K}=t-s}A_{K}=\rho(A)^{-(t-s)}A^{t-s}= uv^{T} + (\rho(A)^{-(t-s)}A^{t-s}-uv^{T}). \] 
		Thus
		\begin{align*}
		&\abs{\mu_{\A}([I] \cap [_{t}J])-\mu_{\A}([I])\mu_{\A}([J])}\\
		&=\rho(A)^{-(s+r)}\abs{\inn{A_{I}(\rho(A)^{-(t-s)}A^{t-s}-uv^{T})A_{J}u,v}}\\
		&\leq \rho(A)^{-(s+r)}\norm{A_{I}^{\ast}v}\norm{A_{J}u}\norm{\rho(A)^{-(t-s)}A^{t-s}-uv^{T}}\\
		&\leq C \beta^{t-s}\rho(A)^{-s}\norm{A_{I}}\rho(A)^{-r}\norm{A_{J}}\\
		&\leq C'\beta^{t-s}\mu_{\A}(I)\mu_{\A }(J) \text{ by Proposition \ref{BasicPropertiesMesaureFullShift}}
		\end{align*}
		where $\beta=\frac{\abs{\lambda_{2}}+\e}{\rho(A)}<1$ for a small $\e>0$ as in Perron-Frobenius theorem \ref{PFThmCones}. Then we have
		\begin{align*}
		\sum_{I ,J}\abs{\mu_{\A}([I] \cap [_{t}J])-\mu_{\A}([I])\mu_{\A}([J])}\leq K\beta^{t-s}\sum_{I,J}\mu_{\A}([I])\mu_{\A}([J])=K\beta^{t-s}.
		\end{align*}
		Hence $\mu_{\A}$ is weak Bernoulli.
	\end{proof}
\end{proposition}

Thus we have proven theorem \ref{Section2thm} 2(a), part 2(b) follows by an approximation argument, see Bowen's book \cite[theorem 1.26]{bowen1975equilibrium}. Finally we end this section with an example which shows that $k$-Gibbs states can be understood in terms of matrices preserving a common cone, for $k$ an even integer.

\begin{example}\label{GeneralizedKusoukaMeasure}
	The following example generalizes the Kusuoka measure (the Kusuoka measure is the case of $k=2$). Let $k$ be an even integer and define
	\[ S=\spn\set{v^{\otimes k}:v \in \R^{d}} \]
	We consider the following cone in $S^{\ast}$
	\[ K=\set{w\in S^{\ast}:\inn{v^{\otimes k},w}_{(\R^{d})^{\otimes k}}\geq 0\text{ for all } v\in \R^{d}} \]
	Note that when $k$ is odd this set is $\set{0}$. When $k$ is even $K$ is a cone with non-void interior (see proposition \ref{PSDTCProp}). The cone $K$ is sometimes referred to as the positive semi-definite tensor cone: in the case of $k=2$ this cone can be identified with positive semi-definite matrices. Suppose that $\A=(A_{0},\ldots , A_{M-1})$ is a collection of matrices with no common proper, non-trivial invariant subspace. Consider the collection $\A'=((A_{0}^{\otimes k})^{\ast}, \ldots , (A_{M-1}^{\otimes k})^{\ast})$. The collection $\A'$ preserves the cone $K$. We claim that in fact $A= \sum_{i}(A_{i}^{\otimes k})^{\ast}$ is irreducible with respect to $K$. To prove this it is enough to show that no eigenvector of $A$ lies on the boundary of $K$ \cite[theorem 4.1]{MR0244284}. Suppose that $w\in K$, $w \neq 0$ and that $Aw=\lambda w$ and define
	\[ W=\spn\set{u: \inn{u^{\otimes k},w}_{(\R^{d})^{\otimes k}}=0} \]
	We claim that $W$ is invariant under $\A$, if $\inn{u^{\otimes k},w}_{(\R^{d})^{\otimes k}}=0$ then
	\begin{align*}
	0=\inn{u^{\otimes k}, Aw}_{(\R^{d})^{\otimes k}}=\sum_{i}\inn{(A_{i}u)^{\otimes k}, w}_{(\R^{d})^{\otimes k}}
	\end{align*}
	as $w \in K$ this implies that $\inn{(A_{i}u)^{\otimes k}, w}_{(\R^{d})^{\otimes k}}=0$ for each $i$. Thus $W$ is $\A$ invariant, so it is either $\R^{d}$ or $\set{0}$. As $w \neq 0$ we must have that $W=\set{0}$. Therefore $w\in \interior(K)$ by lemma \ref{ConeInteriorLemma} and $A$ is irreducible. Constructing the $1$-Gibbs state for $\A'$, we see that it satisfies the Gibbs inequality: there exist constants $C>0$ and $P$ such that
	\[  C^{-1}\mu_{\A'}([x_{0}\cdots x_{n-1}])\leq e^{-nP}\norm{(A_{x_{0}}^{\otimes k})^{\ast}(A_{x_{1}}^{\otimes k})^{\ast}\cdots (A_{x_{n-1}}^{\otimes k})^{\ast}}\leq C \mu_{\A'}([x_{0}\cdots x_{n-1}]). \]
	As $A_{x_{n-1}}^{\otimes k}A_{x_{n-2}}^{\otimes k}\cdots A_{x_{0}}^{\otimes k}= (A_{x_{n-1}}A_{x_{n-2}}\cdots A_{x_{0}})^{\otimes k}$ we have that 
	\[ C^{-1}\mu_{\A'}([x_{0}\cdots x_{n-1}])\leq e^{-nP}\norm{A_{x_{n-1}}A_{x_{n-2}}\cdots A_{x_{0}}}^{k}\leq C \mu_{\A'}([x_{0}\cdots x_{n-1}]). \]
	Strictly speaking the order of the product of matrices is backwards from the Gibbs inequality in equation \eqref{eq:GibbsInequalityMatrices}. By taking $\A= (A_{0}^{\ast}, \ldots , A_{M-1}^{\ast})$ this can be changed (see proposition \ref{irreducibleProp}). Thus we have found an elementary way of constructing $k$-Gibbs states for all even integers.
\end{example}

Therefore we have a completely explicit description of Gibbs states when $t$ is an even integer. 

\section{Transfer operators and exponential mixing}\label{sec:fastmixing}

The goal of this section is to explore a method for constructing matrix Gibbs states and proving ergodic and statistical properties using transfer operators.  This approach is interesting for number of reasons in particular it is an application of transfer operator methods to a problem in sub-additive ergodic theory. It is also a reasonable generalization of example \ref{GeneralizedKusoukaMeasure} using operators on infinite dimensional spaces. We will need the following definitions.

\begin{definition}
	We say that a collection of invertible $d \times d$ matrices $(A_{0},\ldots , A_{M-1})$ is \emph{strongly irreducible} if they do not preserve a finite union of proper and nontrivial subspaces. 
\end{definition}

\begin{definition}
	An element $B \in M_{d}(\R)$ is called \emph{proximal} if $B$ has a simple eigenvalue of modulus $\rho(B)$ and any other eigenvalue has modulus strictly smaller then $\rho(B)$. The collection $(A_{0},\ldots , A_{M-1})$ is called \emph{proximal} if there exists a product $B=A_{x_{0}}\cdots A_{x_{n}}$ that is proximal.
\end{definition}

We have the following theorem.

\begin{theorem}\label{ExpMixingThm}
		Suppose that $\A=(A_{0},\ldots A_{M-1})$ is a collection of real invertible $d \times d$ matrices which is proximal and strongly irreducible. Then for any $t\geq0$ there exists a unique Gibbs state for $(\A,t)$, $\mu_{\A,t}$, moreover
		\begin{enumerate}
			\item 
			$\mu_{\A,t}$ is weak Bernoulli.
			
			\item
			$\mu_{\A ,t}$ has exponential decay of correlations for H\"older continuous functions. That is for a fixed $\theta \in (0,1)$ there are constants $D$ and $\gamma \in (0,1)$ such that
			\[ \abs{\int f\cdot g \circ \sigma^{n}d\mu_{\A,t} - \int f d\mu_{\A,t} \int gd\mu_{\A,t}}\leq D\norm{f}_{\theta}\norm{g}_{\theta}\gamma^{n} \]
			for all $f,g\in \H_{\theta}$, $n \geq 0$.
		\end{enumerate}
\end{theorem}

In the previous section we have seen that the role of the transfer operator for $t=2k$  was played by $A=\sum_{i}A_{i}^{\otimes 2k}$ we need to find a suitable replacement. By identifying $2$-tensors with bilinear forms which are in turn a subspace of the $2$-homogeneous functions one is naturally lead to consider the action of the matrices on $t$-homogeneous functions. This is then equivalent to the action of the matrices on the projective space $\RP^{d-1}$ weighted by the functions $\norm{A_{i}\frac{u}{\norm{u}}}^{t}$. That is, define a transfer operator by
\begin{equation}\label{eq:TransferOperator}
L_{t}f(\ol{u})= \sum_{i=0}^{M-1}\norm{A_{i}\frac{u}{\norm{u}}}^{t}f(\ol{A_{i}u}) 
\end{equation}
which acts on $C(\RP^{d-1})$. The connection between matrix Gibbs states and this operator is made clear in proposition \ref{GibbsStatesTOconstruction}. First we fix some notation. For a function $h$ and a measure $\nu$ we write
\[ \inn{h,\nu}=\int h d\nu. \]
Recall that $\RP^{d-1}$ is obtained by taking the quotient of $\R^{d}\setminus\set{0}$ by the equivalence relation $x \sim y$ if and only if $x= \lambda y$ for some $\lambda \neq 0$. We denote the equivalence class of a vector $v$ by $\ol{v}$. Define a metric on $\RP^{d-1}$ by
\begin{align*}
d(\ol{u},\ol{w})&=\inf \set{\norm{u'-w'}:\norm{u'}=\norm{w'}=1 \text{ and }\ol{u'}=\ol{u},\ol{w'}=\ol{w}}.
\end{align*}

\begin{proposition}\label{GibbsStatesTOconstruction}
	Let $t\geq 0$ and $\A=(A_{0},\ldots , A_{M-1})$ be a collection of invertible matrices. Suppose that there exists $\nu_{t}$ a Borel probability measure not supported on a projective subspace and $h_{t}$ a strictly positive continuous function such that $L_{t}h_{t}=\rho(L_{t})h_{t}$, $L_{t}^{\ast}\nu_{t}=\rho(L_{t})\nu_{t}$ and $\inn{h_{t},\nu_{t}}=1$. Define $L_{i}$ by $L_{i}f(\ol{u})=\norm{A_{i}\frac{u}{\norm{u}}}^{t}f(\ol{A_{i}u})$ then the formula
	\begin{equation}\label{eq:measuredef}
	\mu_{A ,t}[x_{0}x_{1}\cdots x_{n-1}]= \rho(L_{t})^{-n}\int_{\RP^{d-1}}L_{x_{n-1}}\cdots L_{x_{1}}L_{x_{0}} h_{t}(\ol{u}) d\nu_{t}(\ol{u}) 
	\end{equation}
	extends to a shift invariant measure on $\S^{\Z}$. Moreover $\mu_{\A,t}$ is a Gibbs state for $(\A,t)$.
\end{proposition}
\begin{proof}
	The assumption that $h_{t},\nu_{t}$ are eigenvectors corresponding to $\rho(L_{t})$ implies that the formula in \eqref{eq:measuredef} extends to a shift invariant measure by Carath\'eodory's extension theorem. All that remains to be shown is that $\mu_{\A,t}$ satisfies the Gibbs inequality. To see why the Gibbs inequality holds notice
	\[ A \mapsto \int_{\RP^{d-1}}\norm{A\frac{u}{\norm{u}}}^{t}d\nu_{t}(\ol{u}) \]
	is continuous and strictly positive (by the assumption that 
	$\nu_{t}$ is not supported on a projective subspace) from the set of norm one $d \times d$ matrices to $\R$. Take $C>0$ such that 
	\[ \int_{\RP^{d-1}}\norm{A\frac{u}{\norm{u}}}^{t}d\nu(\ol{u}) \geq C\norm{A}^{t} \]
	for all $A \in M_{d}(\R)$. Thus
	\begin{align*}
	\rho(L)^{-n}\inn{L_{x_{n-1}}\cdots L_{x_{1}}L_{x_{0}}h,\nu}\geq (\inf h) C \rho(L)^{-n}\norm{A_{x_{0}}A_{x_{1}} \cdots A_{x_{n-1}} }^{t}
	\end{align*}
	and
	\begin{align*}
	\rho(L)^{-n}\inn{L_{x_{n-1}}\cdots L_{x_{1}}L_{x_{0}}h,\nu}\leq (\sup h) \rho(L)^{-n}\norm{A_{x_{0}}A_{x_{1}} \cdots A_{x_{n-1}} }^{t}.
	\end{align*}
	Which shows that the measure $\mu_{\A,t}$ satisfies the Gibbs inequality.
\end{proof}

If $I=i_{0}i_{1}\cdots i_{n-1}$ we will use the notation that
\[ L_{I}:= L_{i_{n-1}}\cdots L_{i_{1}}L_{i_{0}}. \]
Notice that this is backward from the definition of $A_{I}$. To see why consider
\begin{align*}
L_{x_{1}}L_{x_{0}}f(\ol{u})&= \norm{A_{x_{1}}\frac{u}{\norm{u}}}^{t}L_{x_{0}}f(\ol{A_{x_{1}}u})\\
&=\norm{A_{x_{1}}\frac{u}{\norm{u}}}^{t}\norm{A_{x_{0}}\frac{A_{x_{1}}u}{\norm{A_{x_{1}}u}}}^{t}f(\ol{A_{x_{0}}A_{x_{1}}u})\\
&= \norm{A_{x_{0}}A_{x_{1}}\frac{u}{\norm{u}}}^{t}f(\ol{A_{x_{0}}A_{x_{1}}u}).
\end{align*}
As we can see pre-composition reverses the order of the products.

Operators like $L_{t}$ have appeared frequently in the study of random matrix products. This is however the first time they have been used to construct a measure on $\S^{\Z}$ and deduce ergodic and statistical properties. To prove theorem \ref{ExpMixingThm} all we require is a suitable Perron-Frobenius theorem. For each $\e>0$ denote by $C^{\e}(\RP^{d-1})$ the space of $\e$-H\"older continuous functions in the $d$ metric on $\RP^{d-1}$. This becomes a Banach space in the usual way with norm $\norm{\cdot}_{\e}=\norm{\cdot}_{\infty}+\abs{\cdot}_{\e}$ (where $\abs{f}_{\e}$ is the least $\e$-H\"older constant for $f$). Set $\ol{t}=\min\set{1,t}$. The following theorem is a result of Guivarc'h and Le Page \cite{MR2087783}.

\begin{theorem}[Guivarc'h and Le Page \cite{MR2087783}]\label{PFThm}
	Let $t\geq 0$. Suppose that $(A_{0},\cdots , A_{M-1})$ are real, invertible, strongly irreducible and proximal. Then there exists an $\e$ with $0<\e\leq \ol{t}$ such that the following hold
	\begin{enumerate}
		\item 
		$L_{t}:C^{\e}(\RP^{d-1}) \to C^{\e}(\RP^{d-1})$, that is $L_{t}$ preserves the space of $\e$-H\"older functions.
		
		\item 
		The spectral radius of $L_{t}:C^{\e}(\RP^{d-1})\to C^{\e}(\RP^{d-1})$ is equal to $e^{P(\A,t)}$. That is
		\[ \log \rho(L_{t})=\lim_{n \to \infty}\frac{1}{n}\log \left(\sum_{\abs{I}=n} \norm{A_{I}}^{t}\right)=P(\A,t). \]
		
		\item 
		There exists a unique Borel probability measure $\nu_{t}$ on $\RP^{d-1}$, not supported on a projective subspace, such that $L_{t}^{\ast}\nu_{t}=\rho(L_{t})\nu_{t}$.
		
		\item 
		There exists a unique $\ol{t}$-H\"older function $h_{t}:\RP^{d-1}\to (0,\infty)$ such that $L_{t}h_{t}=\rho(L_{t})h_{t}$ and $\inn{h_{t},v_{t}}=1$.
		
		\item 
		The operator $L_{t}$ has a spectral gap on $C^{\e}(\RP^{d-1})$. That is to say there exists decomposition of $L_{t}$ as $L_{t}=\rho(L_{t})(P_{t}+R_{t})$ where $\rho(R_{t})<1$, $P_{t}R_{t}=R_{t}P_{t}=0$ and
		\[ P_{t}f=\inn{f,\nu_{t}}h_{t} \text{ for all }f \in C^{\e}(\RP^{d-1}). \]
	\end{enumerate}
\end{theorem}
\begin{proof}
	If we take the measure on $GL_{d}(\R)$ to be $\mu = \frac{1}{M}\sum_{i=0}^{M-1}\delta_{A_{i}}$ then the operator called $P^{t}$ in \cite{MR2087783} is a scalar multiple of $L_{t}$ and the result follows from \cite[Theorem 8.8]{MR2087783}. That $h_{t}$ is $\ol{t}$-H\"older is \cite[lemma 4.8]{MR2087783}.
\end{proof}

\begin{corollary}\label{ConvergenceThm}
	Under the assumptions of theorem \ref{PFThm} there exists a constant $C>0$ and $\beta$ with $0<\beta<1$ such that for any $f \in C^{\e}(\RP^{d-1})$ we have
	\[ \norm{\rho(L_{t})^{-n}L_{t}^{n}f-\inn{f,\nu_{t}}h_{t}}_{\e} \leq C \norm{f}_{\e}\beta^{n} \]
	for all $n \geq 0$.
\end{corollary}
\begin{proof}
	Notice that $\rho(L_{t})^{-n}L_{t}^{-n}= P_{t}+R_{t}^{n}$. Thus
	\begin{align*}
	\norm{\rho(L_{t})^{-n}L_{t}^{n}f-\inn{f,\nu_{t}}h_{t}}_{\e}=\norm{R_{t}^{n}f}_{\e}\leq \norm{R_{t}^{n}}_{\e,\text{op}}\norm{f}_{\e}.
	\end{align*}
	Taking $\beta=\rho(R_{t})+\e<1$ for a small $\e>0$ we have the result.
\end{proof}

In order to obtain decay of correlation results we are thus forced into controlling the regularity of $L_{J}h_{t}$. This is the content of the next lemma.

\begin{lemma}\label{regularitylemma}
	\begin{enumerate}
		\item 
		For any $A \in GL_{d}(\R)$ we have that
		\[ d(\ol{Au},\ol{Aw})\leq \frac{2 \norm{A}}{\norm{A\frac{u}{\norm{u}}}}d(\ol{u},\ol{w}). \]
		for all $u,w \in \R^{d}$.
		
		\item 
		For any $A \in GL_{d}(\R)$ and $t\geq 0$ we have that 
		\[ \abs{\norm{A\frac{u}{\norm{u}}}^{t}- \norm{A\frac{w}{\norm{w}}}^{t}}\leq (t+1)\norm{A}^{t}d(\ol{u},\ol{w})^{\ol{t}} \]
		for all $u,w \in \R^{d}$.
		
		\item 
		For any $0<\e\leq \ol{t}$ there exists a constant $K$ such that $\norm{L_{J}h_{t}}_{\e} \leq K \norm{A_{J}}^{t}$ for all $J$.
	\end{enumerate}
\end{lemma}
\begin{proof}
	\begin{enumerate}
		\item
		This is essentially \cite[Lemma 4.6]{MR2087783}. We provide the details for the sake of completeness. Notice for any $u,w$
		\begin{align*}
		\norm{Au}\norm{Aw}\left(\frac{Au}{\norm{Au}}- \frac{Aw}{\norm{Aw}}\right)&= \norm{Aw}Au-\norm{Au}Aw \\
		&=\norm{Aw}Au-\norm{Aw}Aw+\norm{Aw}Aw-\norm{Au}Aw\\
		&=\norm{Aw}(Au-Aw)+(\norm{Aw}-\norm{Au})Aw.
		\end{align*}
		By taking the norm of both sides we have that 
		\[ \norm{Au}\norm{Aw}\norm{\frac{Au}{\norm{Au}}- \frac{Aw}{\norm{Aw}}}\leq 2 \norm{Aw}\norm{A(u-w)}. \]
		Thus
		\begin{align*}
		d(\ol{Au},\ol{Aw})&\leq \norm{\frac{A\frac{u}{\norm{u}}}{\norm{A\frac{u}{\norm{u}}}}- \frac{A\frac{w}{\norm{w}}}{\norm{A\frac{w}{\norm{w}}}}}\\
		&\leq \frac{2}{\norm{A\frac{u}{\norm{u}}}}\norm{A\left(\frac{u}{\norm{u}}-\frac{w}{\norm{w}}\right)}\\
		&\leq \frac{2 \norm{A}}{\norm{A\frac{u}{\norm{u}}}}\norm{\frac{u}{\norm{u}}-\frac{w}{\norm{w}}}.
		\end{align*}
		The same argument holds for $\norm{\frac{-u}{\norm{u}}-\frac{w}{\norm{w}}}$. Hence the result.
		
		\item 
		This is \cite[lemma 4.6]{MR2087783}.
		
		\item 
		Notice
		\begin{align*}
		&\abs{L_{J}h(\ol{u})- L_{J}h(\ol{w})}\\
		&=\abs{\norm{A_{J}\frac{u}{\norm{u}}}^{t}h_{t}(\ol{A_{J}u})-\norm{A_{J}\frac{w}{\norm{w}}}^{t}h_{t}(\ol{A_{J}w}) }\\
		&\leq \norm{A_{J}\frac{u}{\norm{u}}}^{t}\abs{h_{t}(\ol{A_{J}u})-h_{t}(\ol{A_{J}w}) }+ \norm{h_{t}}_{\infty}\abs{\norm{A_{J}\frac{u}{\norm{u}}}^{t}-\norm{A_{J}\frac{w}{\norm{w}}}^{t} }\\
		&\leq \norm{A_{J}\frac{u}{\norm{u}}}^{t}\abs{h_{t}}_{\ol{t}}d(\ol{A_{J}u},\ol{A_{J}w})^{\ol{t}}+ \norm{h_{t}}_{\infty}(t+1)\norm{A_{J}}^{t}d(\ol{u},\ol{w})^{\ol{t}}\\
		&\leq \norm{A_{J}\frac{u}{\norm{u}}}^{t}\abs{h_{t}}_{\ol{t}}\left(\frac{2 \norm{A_{J}}}{\norm{A_{J}\frac{u}{\norm{u}}}}\right)^{\ol{t}}d(\ol{u},\ol{w})^{\ol{t}}+ \norm{h_{t}}_{\infty}(t+1)\norm{A_{J}}^{t}d(\ol{u},\ol{w})^{\ol{t}}\\
		&= \norm{A_{J}\frac{u}{\norm{u}}}^{t-\ol{t}}\norm{A_{J}}^{\ol{t}}\abs{h_{t}}_{\ol{t}}2^{\ol{t}} d(\ol{u},\ol{w})^{\ol{t}}+ \norm{h_{t}}_{\infty}(t+1)\norm{A_{J}}^{t}d(\ol{u},\ol{w})^{\ol{t}}\\
		&\leq \norm{A_{J}}^{t}\abs{h_{t}}_{\ol{t}}2^{\ol{t}} d(\ol{u},\ol{w})^{\ol{t}}+ \norm{h_{t}}_{\infty}(t+1)\norm{A_{J}}^{t}d(\ol{u},\ol{w})^{\ol{t}}\\
		&=\left[\abs{h_{t}}_{\ol{t}}2^{\ol{t}} + \norm{h_{t}}_{\infty}(t+1)\right]\norm{A_{J}}^{t}d(\ol{u},\ol{w})^{\ol{t}}.
		\end{align*}
		Thus for $0<\e\leq \ol{t}$ we have
		\[ \abs{L_{J}h_{t}}_{\e}\leq 2^{\ol{t}-\e}\abs{L_{J}h_{t}}_{\ol{t}}\leq \norm{A_{J}}^{t}2^{\ol{t}-\e}\left[\abs{h_{t}}_{\ol{t}}2^{\ol{t}}+\norm{h_{t}}_{\infty}(t+1)\right]. \]
		Therefore 
		\[ \norm{L_{J}h_{t}}_{\e}=\norm{L_{J}h_{t}}_{\infty}+\abs{L_{J}h_{t}}_{\e}\leq \norm{A_{J}}^{t}\left(\norm{h_{t}}_{\infty}+ 2^{\ol{t}-\e}(\abs{h_{t}}_{\ol{t}}2^{\ol{t}}+\norm{h_{t}}_{\infty}(t+1)) \right). \]
	\end{enumerate}
\end{proof}

The proof of theorem \ref{ExpMixingThm} then follows in exactly the same way as theorem \ref{Section2thm}.

\begin{proof}[proof of theorem \ref{ExpMixingThm}]
	Notice
	\begin{align*}
	&\abs{\mu_{\A,t}([J] \cap \sigma^{-n-\abs{J}}[I])-\mu_{\A,t}([I])\mu_{\A,t}([J])}\\
	&=\abs{\sum_{\abs{K}=n}\mu_{\A,t}([JKI])-\mu_{\A,t}([I])\mu_{\A,t}([J])}\\
	&=\abs{\sum_{\abs{K}=n}\rho(L)^{-(n+\abs{I}+\abs{J})}\inn{L_{I}L_{K}L_{J}h_{t},\nu_{t}}-\rho(L)^{-(\abs{I}+\abs{J})}\inn{L_{I}h_{t},\nu_{t}}\inn{L_{J}h_{t},\nu_{t}}}\\
	&=\rho(L)^{-(\abs{I}+\abs{J})}\abs{\inn{L_{I}\left(\rho(L)^{-n}\sum_{\abs{K}=n}L_{K}\right)L_{J}h_{t},\nu_{t}}-\inn{L_{I}h_{t},\nu_{t}}\inn{L_{J}h_{t},\nu_{t}}}\\
	&=\rho(L)^{-(\abs{I}+\abs{J})}\abs{\inn{L_{I}\rho(L)^{-n}L^{n}L_{J}h_{t},\nu_{t}}-\inn{L_{I}h_{t},\nu_{t}}\inn{L_{J}h_{t},\nu_{t}}}\text{ by \eqref{eq:powersofsum}}\\
	&=\rho(L)^{-(\abs{I}+\abs{J})}\abs{\inn{L_{I}(\rho(L)^{-n}L^{n}L_{J}h_{t}-\inn{L_{J}h_{t},\nu_{t}}h_{t}),\nu_{t}}}\\
	&\leq \rho(L)^{-(\abs{I}+\abs{J})}\norm{L_{I}}_{\infty,\text{op}}\norm{\rho(L)^{-n}L^{n}L_{J}h_{t}-\inn{L_{J}h_{t},\nu_{t}}h_{t}}_{\infty}\\
	&\leq \rho(L)^{-(\abs{I}+\abs{J})}\norm{L_{I}}_{\infty,\text{op}}\norm{L_{J}h_{t}}_{\e}\beta^{n} \text{ by corollary \ref{ConvergenceThm}}\\
	&\leq K \rho(L)^{-(\abs{I}+\abs{J})}\norm{A_{I}}^{t}\norm{A_{J}}^{t}\beta^{n} \text{ by lemma \ref{regularitylemma}}\\
	&\leq C^{2}K \mu_{\A,t}([I])\mu_{\A,t}([J])\beta^{n} \text{ by proposition \ref{GibbsStatesTOconstruction}}
	\end{align*}
	This proves theorem \ref{ExpMixingThm}(1) and (2) follows by an approximation argument as in Bowen's book \cite[Theorem 1.26]{bowen1975equilibrium}.
\end{proof}

Recently in addition to the interest in Gibbs states associated with the norms of matrices there has also been significant interest in the so called singular value potential \cite{MR3820437}, \cite{MR923687}. One can associate a suitable transfer operator to this potential. It seems likely that the method presented in this chapter could be extended to give decay of correlations results for Gibbs states of the singular value potential (in particular taking advantage of \cite[theorem 8.10]{MR2087783}). In addition it seems likely this method could be particularly well suited to studying Gibbs states when $t<0$. For the perspective of thermodynamic formalism it is likely that these measure for $t<0$ are significantly more interesting, for example it is known that the pressure function can fail to analytic \cite{MR2145793} and thus one expects that these systems can exhibit phase transitions. We leave this for future work.

\section{The Weak Bernoulli Property}\label{sec:GeneralMatrices}

The purpose of this section is to prove theorem \ref{MainThm}. The proof is similar to \cite{MR2122435} where scalar potentials satisfying the Bowen property are considered. The key tool is a result of Bradley on $\psi$-mixing sequences of random variables \cite{MR684493} which implies the following lemma.

\begin{lemma}\label{BradleyLemma}
	Let $\mu$ be a shift invariant measure on $\S^{\Z}$. Suppose that for some $N>0$ there exists a constant $C>0$ such that 
	\begin{equation}\label{eq:mixinginequaility}
	C^{-1}\mu([I])\mu([J])\leq \mu([I]\cap \sigma^{-N-\abs{J}} [J])\leq C \mu([I])\mu([J])
	\end{equation}
	for all words $I,J$. Then $\mu$ is weak Bernoulli.
\end{lemma}
\begin{proof}[Proof sketch]
	Notice that for $n\geq N$ we have that
	\begin{align*}
	\mu([I] \cap \sigma^{-n-\abs{J}} [J])&= \sum_{\abs{K}=n-N}\mu([I]\cap \sigma^{-N-\abs{K}-\abs{J}} [KJ])\\
	&\geq C^{-1} \sum_{\abs{K}=n-N}\mu([I])\mu([KJ])\\
	&=C^{-1}\mu([I])\sum_{\abs{K}=n-N}\mu([KJ])\\
	&= C^{-1} \mu([I])\mu([J]).
	\end{align*}
	A similar argument for the other inequality shows that in fact \eqref{eq:mixinginequaility} holds with the same constant $C$ for all $n \geq N$. Thus we have by an approximation argument that 
	\[ \limsup_{n \to \infty}\mu(X \cap \sigma^{-n}Y)\leq C \mu(X)\mu(Y) \]
	and 
	\[ \liminf_{n \to \infty}\mu(X \cap \sigma^{-n}Y) \geq C^{-1}\mu(X)\mu(Y) \]
	for all $X,Y$ Borel measurable. The second inequality gives that $\mu$ is totally ergodic and the first then implies that $\mu$ is mixing by a theorem of Ornstein \cite[Theorem 2.1]{MR0399415}. By an approximation argument we have that 
	\begin{gather*}
	\psi^{\ast}_{n}=\sup\set{\frac{\mu(A \cap B)}{\mu(A)\mu(B)}:A \in \bigvee_{i=n}^{\infty}\sigma^{-i}\P, B \in \bigvee_{i=-\infty}^{-1}\sigma^{-i}\P, \mu(A)\mu(B)>0}\leq C\\
	\psi '_{n}=\inf\set{\frac{\mu(A \cap B)}{\mu(A)\mu(B)}:A \in \bigvee_{i=n}^{\infty}\sigma^{-i}\P, B \in \bigvee_{i=-\infty}^{-1}\sigma^{-i}\P, \mu(A)\mu(B)>0}\geq C^{-1}
	\end{gather*}
	for all $n \geq N$. A result of Bradley \cite[Theorem 1]{MR684493} implies that $\mu$ is $\psi$-mixing; that $\psi$-mixing implies weak Bernoulli is trivial.
\end{proof}

The lemma is essentially a rephrasing of \cite[theorem 4.1(2)]{MR2178042}. With this lemma in hand the proof of theorem \ref{MainThm} is merely an application of the Gibbs inequality. 

\begin{proof}[Proof of theorem \ref{MainThm}]
	Let $N$ be as in the definition of primitive. Let $t>1$ and take $q$ such that $1/t+1/q=1$. Then for any $I,J$
	\begin{align*}
	\mu_{\A,t}([I] \cap \sigma^{-N-\abs{J}}[J])&= \sum_{\abs{K}=N}\mu_{\A,t}([IKJ])\\
	&\geq C^{-1} e^{-(\abs{I}+N+\abs{J})P(\A,t)}\sum_{\abs{K}=N}\norm{A_{I}A_{K}A_{J}}^{t}\\
	&\geq  C^{-1} e^{-(\abs{I}+N+\abs{J})P(\A,t)}M^{-Nt/q}\left(\sum_{\abs{K}=N}\norm{A_{I}A_{K}A_{J}}\right)^{t}\\
	&\geq  C^{-1} e^{-(\abs{I}+N+\abs{J})P(\A,t)}M^{-Nt/q}\delta^{t} \norm{A_{I}}^{t}\norm{A_{J}}^{t}\\
	&\geq  C^{-2} e^{-N P(\A,t)}M^{-Nt/q}\delta^{t} \mu_{\A ,t}([I])\mu_{\A,t}([J])
	\end{align*}
	where $M=\abs{\Sigma}$. For $0<t\leq 1$ we have that 
	\begin{align*}
	\mu_{\A,t}([I] \cap \sigma^{-N-\abs{J}}[J])&= \sum_{\abs{K}=N}\mu_{\A,t}([IKJ])\\
	&\geq C^{-1} e^{-(\abs{I}+N+\abs{J})P(\A,t)}\sum_{\abs{K}=N}\norm{A_{I}A_{K}A_{J}}^{t}\\
	&\geq  C^{-1} e^{-(\abs{I}+N+\abs{J})P(\A,t)}\left(\sum_{\abs{K}=N}\norm{A_{I}A_{K}A_{J}}\right)^{t}\\
	&\geq  C^{-1} e^{-(\abs{I}+N+\abs{J})P(\A,t)}\delta^{t} \norm{A_{I}}^{t}\norm{A_{J}}^{t}\\
	&\geq  C^{-2} e^{-N P(\A,t)}\delta^{t} \mu_{\A ,t}([I])\mu_{\A,t}([J]).
	\end{align*}
	For matrix Gibbs states the right hand inequality in equation \eqref{eq:mixinginequaility} always holds. This is a simple consequence of the Gibbs inequality and the fact that the norm is sub-multiplicative, it was noticed in \cite{morris_2017}. The result then follows from lemma \ref{BradleyLemma}.
\end{proof}

\begin{acknowledgements}
	I am grateful to Anthony Quas for suggesting this problem as well as many helpful discussions and comments on drafts of this manuscript. I am also grateful to Chris Bose for helpful discussions and comments on drafts of this manuscript. I am grateful to the referee for a careful reading of the paper, providing helpful suggestions and for pointing me to \cite{MR2087783} which significantly improved the contents of section \ref{sec:fastmixing}.
\end{acknowledgements}

\section{Appendix}

Here we collect some facts as well as the basic definitions and properties of cones in finite dimensional vector spaces. Most of material on cones will be familiar from the classical Perron-Frobenius theory for non-negative matrices. In addition we collect some elementary propositions and lemmas which we use in the article.

\begin{definition}
	A subset $K \subseteq \R^{d}$ is called a cone if 
	\begin{enumerate}
		\item 
		$K \cap (-K) = \set{0}$
		
		\item
		$\lambda K =K$ for all $\lambda >0$
		
		\item
		$K$ is convex
	\end{enumerate}
\end{definition}

If $K$ is a cone then we define the dual of $K$ by
\[ K^{\ast}:= \set{w:\inn{u,w}\geq 0 \text{ for all }u \in K}. \]

\begin{definition}
	Let $A:\R^{d} \to \R^{d}$ be a linear map and $K$ be cone.
	\begin{itemize}
		\item 
		We say that $A$ is $K$-non-negative provided $AK \subseteq K$ and we write $A\geq^{K}0$.
		
		\item 
		We say that $A$ is $K$-positive if $A(K\setminus \set{0}) \subseteq \interior(K)$ and we write $A> ^{K}0$.
		
		\item 
		We say $A\geq^{K}0$ is $K$-primitive if there exists an $N$ such that $A^{N}$ is $K$-positive.
		
		\item 
		We say $A\geq^{K}0$ is $K$-irreducible if $\sum_{k=0}^{d-1}A^{k}$ is $K$-positive.
	\end{itemize}
	Often when $K$ is understood we suppress the notation.
\end{definition}

There are various definitions of $K$-irreducible in the literature. It is known that these conditions are all equivalent however finding a complete proof in the literature is surprising difficult. Thus for the sake of completeness we include proposition \ref{IrrDefs}. In order to achieve $3 \implies 4$ we need the following lemma.

\begin{lemma}\label{ConeInteriorLemma}
	Suppose that $K\subseteq \R^{d}$ is a cone. Then
	\[ \interior(K)=\set{u:\inn{u,v}>0 \text{ for all }v \in K^{\ast}\setminus \set{0}}. \]
\end{lemma}
\begin{proof}
	First we recall that 
	\[ K=\set{u: \inn{u,v}\geq 0 \text{ for all }v \in K^{\ast} \text{ with }\norm{v}=1 } \]
	this is a very general fact for closed cones in Banach spaces which follows from a suitable Hahn-Banach theorem, see \cite{eveson_nussbaum_1995} for a nice discussion. Let 
	\[u \in\set{u:\inn{u,v}>0 \text{ for all }v \in K^{\ast}\setminus \set{0}}.\]
	Now take $\delta>0$ such that $\inn{u,v}>\delta$ for all $v \in K^{\ast}$ with $\norm{v}=1$. Suppose that $\norm{w-u}<\delta/2$ then for any $v \in K^{\ast}$ with $\norm{v}=1$ we have
	\begin{align*}
	\abs{\inn{u,v}-\inn{w,v}}= \abs{\inn{u-w,v}}\leq \norm{w-u}<\delta/2.
	\end{align*}
	This implies that $\inn{w,v}\geq \delta/2>0$ and thus $w \in K$. As $B(u,\delta/2)\subseteq K$ we have that $u \in \interior(K)$. 
	
	Now suppose that $u \in \interior(K)$. Take $\delta>0$ such that $B(u,\delta)\subseteq K$ if $\norm{w}=\delta/2$ then 
	\[ \norm{(u+w)-u}=\norm{w}<\delta \text{ and }\norm{(u-w)-u}=\norm{w}<\delta \]
	implies that $u-w,u+w \in K$. Thus for any $v \in K^{\ast}$
	\[0 \leq \inn{u+w,v}=\inn{u,v}+\inn{w,v} \text{ and }0 \leq \inn{u-w,v}=\inn{u,v}-\inn{w,v} \]
	which implies 
	\[ -\inn{w,v}\leq \inn{u,v}\leq \inn{w,v}. \]
	Then we have
	\[\norm{v}=2/\delta\sup_{\norm{w}= \delta/2}\abs{\inn{w,v}}\leq 2/\delta \inn{u,v}. \]
	In particular if $v\neq 0$ then $\inn{u,v}>0$.
\end{proof}

\begin{proposition}\label{IrrDefs}
	Suppose that $A$ preserves a non-void cone $K\subseteq \R^{d}$. The following are equivalent:
	\begin{enumerate}
		\item 
		$A$ has no eigenvector contained in $\partial K$.
		
		\item 
		$A$ has no invariant faces.
		
		\item 
		$(I+A)^{d-1}$ is $K$-positive.
		
		\item 
		$\sum_{k=0}^{d-1}A^{k}$ is $K$-positive.
	\end{enumerate}
\end{proposition}
\begin{proof}
	($1 \iff 2$) is \cite[theorem 4.2]{MR0244284}.\\
	
	($2 \implies 3$) is \cite[lemma 4.2]{MR0244284}.\\
	
	($4 \implies 1$) is clear if $A$ has an eigenvector contained in $\partial K$ then so does $\sum_{k=0}^{d-1}A^{k}$.\\
	
	($3 \implies 4$) Suppose that $u \in K\setminus \set{0}$ by the assumption that $(I+A)^{d-1}>^{K}0$ we have that for any $v \in K^{\ast}\setminus \set{0}$
	\[ 0<\inn{(I+A)^{d-1}u,v}= \sum_{k=0}^{d-1}\binom{d-1}{k}\inn{A^{k}u,v} \]
	by lemma \ref{ConeInteriorLemma}. This implies that
	\[ 0<\sum_{k=0}^{d-1}\inn{A^{k}u,v}= \inn{\left(\sum_{k=0}^{d-1}A^{k} \right)u,v}  \]
	and hence $\sum_{k=0}^{d-1}A^{k}u \in \interior(K)$ by lemma \ref{ConeInteriorLemma} and $\sum_{k=0}^{d-1}A^{k}$ is $K$-positive.
\end{proof}

The Perron-Frobenius theorem holds for abstract finite dimensional cones just as it does for the positive quadrant.

\begin{theorem}\label{PFThmCones}
	Suppose that $K$ is a closed cone with non-void interior.
	\begin{enumerate}
		\item
		If $A$ is $K$-non-negative then
		\begin{enumerate}
			\item
			$\rho(A)$ is an eigenvalue.
			
			\item
			$K$ contains an eigenvector corresponding to $\rho(A)$.
		\end{enumerate}
		\item
		If $A$ is $K$-irreducible then
		\begin{enumerate}
			\item
			$\rho(A)$ is a simple eigenvalue, and any other eigenvalue with the same modulus is simple.
			
			\item 
			Suppose that $u$ is an eigenvector for $A$ corresponding to $\rho(A)$ and $v$ is an eigenvector of $A^{T}$ corresponding to $\rho(A)$ normalized so that $\inn{u,v}=1$. Then
			\[ \lim_{n \to \infty}\frac{1}{n}\sum_{k=1}^{n}\rho(A)^{-k}A^{k}=P \]
			where $Pw=\inn{w,v}u$.
		\end{enumerate}
		
		\item
		If $A$ is $K$-primitive then 
		\begin{enumerate}
			\item
			$\rho(A)$ is a simple eigenvalue, which is greater in modulus then any other eigenvalue.
			
			\item 
			Suppose that $u$ is an eigenvector for $A$ corresponding to $\rho(A)$ and $v$ is an eigenvector of $A^{T}$ corresponding to $\rho(A)$ normalized so that $\inn{u,v}=1$. Then for all small $\e>0$ there exists $C>0$ such that for all $n \geq 0$
			\[ \norm{\rho(A)^{-n}A^{n}-P}\leq C \left(\frac{\abs{\lambda_{2}}+\e}{\rho(A)}\right)^{n} \]
			where $Pw=\inn{w,v}u$.
		\end{enumerate}
	\end{enumerate}
\end{theorem}
\begin{proof}
	\begin{enumerate}
		\item 
		This is \cite[theorem 3.1]{MR0244284}.
		
		\item 
		(a) can be found in \cite[theorem 4.3]{MR0244284}. (b) follows from (a) using the same proof as for non-negative matrices, which can be found in \cite[theorem 8.6.1]{MR2978290}.
		
		\item 
		This result is well known. It can be proved for example using Hilbert's projective metric and holds in significant generality see for example \cite[theorem 2.3]{eveson_nussbaum_1995} (although the result in \cite{eveson_nussbaum_1995} is significantly more powerful then needed here). The article \cite{MR0244284} contains a proof this result when $A$ is assumed $K$-positive.
	\end{enumerate}
\end{proof}

It is clear from the definition of irreducible and primitive that the eigenvector corresponding to $\rho(A)$ is contained in the interior of the cone $K$. This agrees with the fact from classical Perron-Frobenius theory the the eigenvector has all positive entries. Notice also that for $K$-irreducible/primitive matrices there always exist vectors $u$ and $v$ with $u$ an eigenvector for $A$ corresponding to $\rho(A)$ and $v$ is an eigenvector for $A^{T}$ corresponding to $\rho(A)$ such that $\inn{u,v}=1$ (by 1(b) and the observation that $u \in \interior(K)$ by irreducibility/primitivity). This ensures that 2(b) and 3(b) are never vacuous. We need the following to produce the Gibbs inequality.

\begin{lemma}\label{approxnormlemma}
	Suppose that $K$ is a cone and that $D\subset \interior(K)$ and $D^{\ast}\subset \interior(K^{\ast})$ are non-empty and compact. Then there exists a constant $C>0$ such that
	\[ C^{-1}\norm{A}\leq \inn{Au,v}\leq C \norm{A} \]
	for all $u\in D$, $v \in D^{\ast}$, and $A\geq^{K}0$.
\end{lemma}
\begin{proof}
	Suppose $A\geq^{K}0$, and that for some $u \in \interior(K)$ and $v \in \interior(K^{\ast})$ we have $\inn{Au,v}=0$. Then $Au=0$ by an arguement similar to lemma \ref{ConeInteriorLemma}. Thus for any $w \in K^{\ast}$ we have that $\inn{u,A^{\ast}w}=0$ which implies that $A^{\ast}w=0$ by lemma \ref{ConeInteriorLemma}. Therefore $A^{\ast}=0$ and of course $A=0$. Thus the function
	\[ (A,u,v)\mapsto \inn{Au,v} \]
	is continuous and $\inn{Au,v}>0$ for all $A\geq^{K}0$, $A \neq 0$ and $u\in D,v\in D^{\ast}$. As the set of norm one $K$ non-negative matrices cross $D \times D^{\ast}$ is compact we can find a $C>0$ such that
	\[ C^{-1}\leq \inn{\frac{Au}{\norm{A}},v}\leq C \]
	for all $A\geq^{K}0$, $A \neq 0$ and $u\in D, v\in D^{\ast}$. Clearly the inequality holds for $A=0$ hence we have the result.
\end{proof}

With the proceeding lemma the proof of the following proposition which relates the definition of irreducibility and primitivity from the introduction to that for operators is straightforward.

\begin{proposition}\label{IrrPrimDefProp}
	Let $\A=(A_{0},\ldots , A_{M-1})\in M_{d}(\R)$ and define $L_{\A}$ as in example \ref{ClassicalKusuokaMeasure}. 
	\begin{enumerate}
		\item 
		If $L_{\A}$ is irreducible then $\A$ satisfies equation \eqref{eq:irrinequality}.
		
		\item 
		If $L_{\A}$ is primitive then $\A$ satisfies equation \eqref{eq:priminequality}.
	\end{enumerate}
\end{proposition}
\begin{proof}
	We will prove (2); then (1) will be similar. Let $L_{i}$ be as in example \ref{ClassicalKusuokaMeasure}. Take $N$ such that $L_{\A}^{N}>^{K}0$ and $U,V$ positive definite matrices. Set $D=\set{U}$ and $D^{\ast}=(L_{\A}^{\ast})^{N}(\set{W\in K^{\ast}:\inn{U,W}_{\text{HS}}=1})$. Notice that $\set{W\in K^{\ast}:\inn{U,W}_{\text{HS}}=1}$ is closed and bounded (by lemma \ref{approxnormlemma}) hence compact. Take $C>0$ as in lemma \ref{approxnormlemma} then
	\begin{align*}
	\sum_{\abs{K}=N} \norm{A_{I}A_{K}A_{J}}^{2}&\geq C^{-1} \inn{L_{I}L_{\A}^{N}L_{J}U,V}_{\text{HS}}\\
	&= C^{-1} \inn{L_{I}U,V}_{\text{HS}} \inn{L_{J}U,(L_{\A}^{\ast})^{N}\frac{L_{I}^{\ast}V}{\inn{L_{I}U,V}}}_{\text{HS}}\\
	&\geq C^{-3} \norm{A_{I}}^{2}\norm{A_{J}}^{2}.
	\end{align*}
	Where we have used the fact that $\norm{L_{I}}=\norm{A_{I}}^{2}$. Therefore 
	\[ \sum_{\abs{K}=N}\norm{A_{I}A_{K}A_{J}}\geq \left(\sum_{\abs{K}=N} \norm{A_{I}A_{K}A_{J}}^{2}\right)^{1/2} \geq C^{-3/2} \norm{A_{I}}\norm{A_{J}}. \]
\end{proof}

\begin{proposition}\label{PSDTCProp}
	Let $k$ be an even number and define
	\[ S=\spn\set{v^{\otimes k}:v \in \R^{d}} \]
	and
	\[ K=\set{w\in S^{\ast}:\inn{v^{\otimes k},w}_{(\R^{d})^{\otimes k}}\geq 0\text{ for all }v\in \R^{d}}. \]
	Then $K$ is a closed cone with non-void interior.
\end{proposition}
\begin{proof}
	That $K$ is a closed cone is trivial. Thus we turn our attention to showing that $K$ has a non-void interior. First we note that there exist elements $w \in K$ such that $\inn{v^{\otimes k},w}>0$ for all $v \in \R^{d}\setminus \set{0}$. For example define a multi-linear map $f:(\R^{d})^{k}\to \R$ by 
	\[ f(v^{1},v^{2}, \ldots , v^{k})=\sum_{i=1}^{d} v^{1}_{i}v^{2}_{i}\cdots v^{k}_{i} \]
	this gives a linear map $f:(\R^{d})^{\otimes k}\to \R$ such that
	\[ f(v^{\otimes k})=\sum_{i=1}^{d}v_{i}^{k}>0. \]
	Now if $v^{n}\xrightarrow{n \to \infty} w$ then 
	\begin{align*}
	(v^{n})^{\otimes k}&=\sum_{i_{1}\cdots i_{k}}v^{n}_{i_{1}}\cdots v^{n}_{i_{k}}e_{i_{1}}\otimes \cdots \otimes e_{i_{k}} \\
	&\xrightarrow{n \to \infty} \sum_{i_{1}\cdots i_{k}}w_{i_{1}}\cdots w_{i_{k}}e_{i_{1}}\otimes \cdots \otimes e_{i_{k}}\\
	&=w^{\otimes k}.
	\end{align*}
	Thus $\set{v^{\otimes k}:\norm{v^{\otimes k}}=1}$ is compact. Take $\delta>0$ such that $f(v^{\otimes k})>\delta$ for all $v$ for which $\norm{v^{\otimes k}}=1$. Now if $g \in S^{\ast}$ is such that $\norm{f-g}<\delta/2$ then $g \in K$. Hence $\interior(K)\neq \emptyset$.
\end{proof}

\begin{proposition}\label{irreducibleProp}
	Suppose that $(A_{0},\ldots , A_{M-1})$ is irreducible then $(A_{0}^{\ast}, \ldots , A_{M-1}^{\ast})$ is also irreducible.
\end{proposition}
\begin{proof}
	Notice that if $A_{i}^{\ast}W \subseteq W$ then $A_{i}W^{\perp} \subseteq W^{\perp}$. To see this consider for any $u \in W^{\perp}$ and $w \in W$ we have
	\[ 0=\inn{A_{i}^{\ast}w,u}= \inn{w,A_{i}u} \]
	which implies that $A_{i}u \in W^{\perp}$. If $A_{i}W \subseteq W$ for all $0 \leq i \leq M-1$ then $W^{\perp}=\set{0}$ or $\R^{d}$ hence $W=\set{0}$ or $\R^{d}$.
\end{proof} 

\bibliography{MESBib}{}
\bibliographystyle{abbrv}

\end{document}